\documentclass[12pt,reqno,a4paper]{amsart}

\usepackage[plain]{fullpage}

\usepackage{calrsfs}
\usepackage[OT2,T1]{fontenc}
\usepackage[english]{babel}
\usepackage{newtxtext}
\usepackage{newtxmath}
\usepackage{bbm}
\usepackage{tikz}
\usepackage{booktabs}
\usepackage[table]{xcolor}
\usepackage{caption}
\usepackage{longtable}
\usepackage{enumitem}

\usepackage{mathtools}
\usepackage{mathrsfs}
\usepackage[unicode=true,pdfusetitle, bookmarks=true,bookmarksnumbered=false,bookmarksopen=false, breaklinks=true,pdfborder={0 0 0},pdfborderstyle={},backref=false,colorlinks=true]{hyperref}
\definecolor{myblue}{rgb}{0.09,0.52,0.74}
\definecolor{mygreen}{rgb}{0.05,0.6,0.2}
\hypersetup{pdfborder={0 0 0},pdfborderstyle={},colorlinks=true,linkcolor=myblue,citecolor=mygreen,urlcolor=blue}

\theoremstyle{plain}
\newtheorem{introtheorem}{Theorem}

\newtheorem{introconjecture}[introtheorem]{Conjecture}

\newtheorem{theorem}{Theorem}[section]

\newtheorem{conjecture}[theorem]{Conjecture}

\theoremstyle{definition}
\newtheorem{definition}[theorem]{Definition}
\newtheorem{introdefinition}[introtheorem]{Definition}

\newtheorem{example}[theorem]{Example}
\theoremstyle{remark}
\newtheorem{remark}[theorem]{Remark}
\newtheorem{introremark}[introtheorem]{Remark}

\counterwithin{equation}{section}

\newcommand{\Z}{\mathbf{Z}}

\newcommand{\R}{\mathbf{R}}
\newcommand{\Q}{\mathbf{Q}}

\renewcommand{\C}{\mathbf{C}}

\newcommand{\mr}{\mathrm}
\renewcommand{\d}{\mathrm{d}}
\newcommand{\ii}{\mathrm{i}}
\newcommand{\T}{\mathbf{T}}
\renewcommand{\geq}{\geqslant}
\renewcommand{\ge}{\geqslant}
\renewcommand{\leq}{\leqslant}
\renewcommand{\le}{\leqslant}
\renewcommand{\bar}{\overline}
\renewcommand{\epsilon}{\varepsilon}

\DeclareMathOperator{\Gal}{Gal}

\DeclareMathOperator{\m}{\mathrm{m}}

\DeclareMathOperator{\ord}{ord}

\definecolor{amethyst}{rgb}{0.6, 0.4, 0.8}
\definecolor{darkcyan}{rgb}{0.0, 0.55, 0.55}

\DeclarePairedDelimiter\abs{\lvert}{\rvert}%

\makeatletter
\let\oldabs\abs
\def\abs{\@ifstar{\oldabs}{\oldabs*}}

\makeatletter
\@namedef{subjclassname@1991}{\textup{2020} Mathematics Subject Classification}
\makeatother

\begin{document}

\date{\today}
\title{Relating Mahler measures and Dirichlet $L$-values: \\ new evidence for Chinburg's conjectures}

\author[D.~Hokken]{David Hokken}
\address{\normalfont DH: Mathematisch Instituut\\
        Universiteit Utrecht\\
        Postbus 80.010, 3508 TA Utrecht, Nederland}
\email{{\tt d.p.t.hokken@uu.nl}}

\author[M.~Mehrabdollahei]{Mahya Mehrabdollahei}
\address{\normalfont MM: Max Planck Institute for Mathematics, Vivatsgasse 7, 53111 Bonn, Germany}
\email{{\tt mehrabdollahei@mpim-bonn.mpg.de}}

\author{Berend Ringeling}
\address{\normalfont BR: D\'epartement de math\'ematiques et de statistique\\
Universit\'e de Montr\'eal\\
CP 6128 succ.~Centre-Ville\\
\indent Montr\'eal, QC H3C 3J7\\
Canada}
\email{{\tt bjringeling@gmail.com}}


\begin{abstract}
Let $\chi_{-f}$ be the odd quadratic Dirichlet character of conductor $f$, and let $\m(P)$ denote the Mahler measure of a polynomial $P$. In 1984, Chinburg conjectured that for any such $\chi_{-f}$ there exist an integral bivariate rational function $P$ (and, in the strong form, an integral polynomial) such that $\m(P)$ is a rational multiple of $L'(\chi_{-f},-1)$. The strong form of the conjecture was previously known to hold for $18$ values of $f$. We double the number of numerical examples, giving $8$ new instances of the strong and $18$ new instances of the weak conjecture. Our examples arise from an explicit approach, which also captures almost all of the previously known results, and is based on work of Boyd and Rodriguez-Villegas. Moreover, we prove Chinburg's weak conjecture if we allow cyclotomic coefficients.
\end{abstract}

\maketitle

\section{Introduction}

Let $\T^n$ be the complex unit torus in $n$ dimensions. The \emph{(logarithmic) Mahler measure} of a rational function $P \in \C(x_1, \ldots, x_n)^{\times}$ is the real number
\begin{equation*}
\m(P) \coloneqq \frac{1}{(2\pi \ii)^n}\int_{\T^n}\log|P(x_1,\dots, x_n)|\frac{\d x_1}{x_1}\cdots \frac{\d x_n}{x_n}.
\end{equation*}
Let $-f<0$ be a fundamental discriminant and denote by $\chi_{-f} \coloneqq \left(\frac{-f}{\cdot}\right)$ the odd, primitive, quadratic Dirichlet character of conductor $f$, with $L$-function $L(\chi_{-f}, s)$. In 1981, Smyth \cite{Smyth} proved the surprising equality
\begin{equation*}
\m(1+x+y) = L'(\chi_{-3}, -1)
\end{equation*}
and in the years since many more fascinating relations between multivariate Mahler measures and $L$-values of arithmetic objects have been established (see, e.g., \cite{Boyd-L,Deninger} and the survey \cite[Section 8.4]{BZ}). Here, in the spirit of Smyth's first example, we study rational relations between Mahler measures of bivariate polynomials and Dirichlet $L$-values. In this connection, Chinburg posed the following conjectures in a 1984 unpublished manuscript \cite{Chinburg}.

\begin{introconjecture}[Chinburg's conjectures]
\label{con:Chinburg} \ 
\begin{itemize}
    \item[\textup{(1)}] \textup{(weak form)} For any fundamental discriminant $-f<0$, there exists a bivariate rational function $P \in \Z(x,y)$ and a rational number $r \in \Q^{\times}$ such that $\m(P) = r L'(\chi_{-f}, -1).$
    \item[\textup{(2)}] \textup{(strong form)} In \textup{(1)}, we may take $P \in \Z[x,y]$.
\end{itemize}
\end{introconjecture}

The strong form of Conjecture~\ref{con:Chinburg} was previously known to hold for the $18$ values
\begin{equation}
\label{eq:oldexamples}
f=3, 4, 7, 8, 11, 15, 19, 20, 23, 24, 35, 39, 40, 55, 84, 120, 303, 755;
\end{equation}
see Table~\ref{tab:previous_strong}. We find new numerical, unproven examples of Conjecture~\ref{con:Chinburg} for the conductors
\begin{equation}
\label{eq:newexamples}
\begin{aligned}
    f = 43, 52, 56, &\, 68, 111, 132, 228, 696 \quad \text{(strong);} \\
    f = 31, 47, 51, 59, 67, 87, 88, 107, 123, &\, 136, 164, 183, 195, 260, 264, 291, 555, 804 \quad \text{(weak),}
\end{aligned}
\end{equation}
see Table~\ref{tab:new_strong} and Table~\ref{tab:T1}. All examples have been verified to $500$ digits precision. 

\begin{introremark}
\label{rem:Chinburg}
Let us clarify a few points concerning Conjecture~\ref{con:Chinburg}. Firstly, Chinburg stated only the weak form as a conjecture, and the strong form as a question, but subsequent literature \cite{Boyd-L,BRV1,BRV2,LQ,Ray} has referred to `Chinburg's conjecture' as the strong form of Conjecture~\ref{con:Chinburg} --- likely because Chinburg's manuscript contains a proof of the weak form. This also explains the absence of weak examples in the literature. However, at least since 2021 it is known that there is an irreparable error in that proof (see also \cite{Zudilin}), namely in Lemma $1$ of Chinburg's manuscript \cite{Chinburg}. At present, therefore, both the weak and strong form of Conjecture~\ref{con:Chinburg} remain open. Secondly, Chinburg actually posed his conjecture and question in the following more general way: given a fundamental discriminant $-f < 0$ and an odd integer $n \geq 0$, does there exist a polynomial (or rational function) $P$ in $n+1$ variables with integer coefficients such that $\m(P) = rL'(\chi_{-f}, -n)$ for some $r \in \Q^{\times}$? Here, we only study the case $n=1$.
\end{introremark}

\subsection*{Relations to the literature}

Our results (see \eqref{eq:newexamples}) follow from a synthesis of the existing quadratic bivariate examples $P=a(x)(y^2+1)+b(x)y \in \Z[x,y]$ giving rise to examples of Conjecture~\ref{con:Chinburg}, so let us first discuss the relevant literature.
There are several earlier works concerned with expressing Mahler measures of bivariate rational functions as $\Q$-linear combinations of one or more $L$-values of odd, quadratic Dirichlet characters. Table~\ref{tab:previous_strong} lists previously discovered polynomials giving examples of Chinburg's strong conjecture; only one example per conductor is listed. 

Ray \cite{Ray} defines, for each odd quadratic Dirichlet character $\chi$, one explicit bivariate polynomial and examines its Mahler measure by means of complex analysis and certain twisted $L$-series. As he shows, his method produces examples of Chinburg's conjecture for conductors $f \leq 24$ only.

The examples of Boyd and Rodriguez-Villegas \cite{BRV1} come from a systematic study of polynomials of the form $A(x,y) = p(x)y-q(x)$ with $p$ and $q$ cyclotomic; such polynomials were already considered by Boyd \cite[\S 1B]{Boyd-L}. Boyd and Rodriguez-Villegas also have a second paper on the topic \cite{BRV2}, where they generalize their previous work to a wide class of polynomials. This class contains certain bivariate polynomials of the form $P(x,y) = a(x)(y^2+1)+b(x)y$ defining curves of genus $0$ or $1$, which is also the type of polynomials that we will study in this work.
In this relation, we also mention the work of Vandervelde \cite{Vandervelde}, who similarly studied certain bivariate polynomials that define genus $0$ curves and provided an explicit formula for the Mahler measure of such polynomials. 
Later, Guilloux and Marché \cite{GM} reformulated the work of \cite{BRV2} in modern language, and gave an explicit formula for the Mahler measure of the polynomials studied in \cite{BRV2}.
Liu and Qin \cite{LQ} settled the strong Chinburg conjecture for three more conductors as part of a wide-ranging numerical study of the relation between special $L$-values and Mahler measures of certain families of polynomials defining genus $1$, $2$ or $3$ curves; their results related to Conjecture~\ref{con:Chinburg} follow from \cite{BRV2}.

Bertin and the second author \cite{BM} studied another family of bivariate polynomials, unbounded both in degree and genus, producing another example of the weak form of Conjecture~\ref{con:Chinburg} for conductor $15$; their example has explicit `rational correction factors'. Furthermore, they formulate, and provide evidence for, a further generalization to all odd, primitive Dirichlet characters of the general form of Chinburg's conjectures discussed in Remark~\ref{rem:Chinburg}. 

\subsection*{Approach and outline of the paper}

We start with the following definition.

\begin{introdefinition}
\label{def:Chinburg-polynomial}
Define $d_f \coloneqq L'(\chi_{-f},-1)$. We say that $P \in \Z[x,y]$ is a \emph{Chinburg polynomial} if there are fundamental discriminants $-f_1, \ldots, -f_s < 0$ and rational numbers $r_1, \ldots, r_s \in \Q$ such that
\begin{equation*}
    \m(P) = r_1 d_{f_1} + \cdots + r_s d_{f_s}.
\end{equation*}
\end{introdefinition}
If $P$ is a Chinburg polynomial and $s$ as above equals $1$, we immediately obtain an example of the strong Chinburg conjecture. When $s>1$ and the weak Chinburg conjecture is already known for at least $s-1$ of the $\{f_{1}, \ldots, f_{s} \}$, we may use linear algebra and the property $\m(PQ) = \m(P) + \m(Q)$ to obtain an example of the weak (and perhaps even the strong) Chinburg conjecture by combining various such polynomials $P$. Hence our goal is to construct Chinburg polynomials.

In \S\ref{sec:permissible-polynomials}, we formulate the notion of \emph{permissible polynomials}. Based on extensive numerical experiments, we conjecture that any permissible polynomial is a Chinburg polynomial; see Conjecture~\ref{con:ourcon} below. Then we show that almost all existing examples in the literature on Chinburg's conjectures come from such permissible polynomials; this is one of the main contributions of our work. In Remark~\ref{rem:conjecture}, we discuss why we believe that Conjecture~\ref{con:ourcon} holds. In \S \ref{sec:tables}, we provide an overview of old and new permissible polynomials; all tables with results are included in the appendix. We also discuss (and provide) the \texttt{Sage} algorithm used to obtain these examples. In \S \ref{sec:algebraicChinburg}, we prove Chinburg's weak conjecture in the setting of rational functions with coefficients in an abelian extension of $\Q$, answering a question of Brunault.

\subsection*{Notation}

For $a, b \in \R$, we write $a \doteq b$ if the first $500$ decimal digits of $a$ and $b$ coincide. For a fundamental discriminant $-f<0$, we denote $d_f \coloneqq L'(\chi_{-f},-1)$.

\subsection*{Acknowledgements}
The authors would like to thank François Brunault for discussions, feedback and for posing the question adressed in~\S\ref{sec:algebraicChinburg}. Furthermore, the authors would like to thank Marie-José Bertin, Gunther Cornelissen, Antonin Guilloux, Matilde Lalín, Riccardo Pengo, Thu-Hà Trieu, and Wadim Zudilin for discussions and feedback, and Steven Charlton for providing a polynomial that settles the strong Chinburg conjecture for conductor $68$.

D.~H. is supported by the Dutch Research Council (NWO) through the grant OCENW.M20.233.

D.~H. and B.~R. would like to thank their hosts for pleasant stays at the University of G\"ottingen in March 2024 and June 2025 and at Universit\'e de Montr\'eal in September 2024.

M.~M. gratefully acknowledges the support and hospitality of Utrecht University, Universit\'e de Montr\'eal and of the Max Planck Institute for Mathematics in Bonn, and thanks Pieter Moree and Don Zagier for facilitating these research visits and for many fruitful discussions.

\section{Permissible polynomials}
\label{sec:permissible-polynomials}

In this section, we consider bivariate polynomials of the form 
\begin{equation*}
P(x,y) = a(x)(y^2+1)+b(x)y \in \Z[x,y]
\end{equation*}
and define when such a polynomial is \emph{permissible}. We start by setting up some notation and definitions.

\begin{definition}
The \emph{discriminant} of $P$ (with respect to $y$) is the univariate polynomial
\begin{equation*}
\Delta(P)(x) = \Delta_y(P)(x) \coloneqq b(x)^2-4a(x)^2.
\end{equation*}
Furthermore, we define polynomials $g,h \in \Z[x]$ uniquely by the property that $\Delta(P)(x) = gh^2$ with $g$ squarefree.   
\end{definition}

\begin{definition}
A univariate polynomial $p(x) \in \R[x]$ is \emph{reciprocal} if $x^{\deg{p}} p(1/x) = p(x)$. Furthermore, for an integer $\ell$, we say that $p(x)$ is an \emph{$\ell$-shifted reciprocal polynomial} if $x^{\ell} p(1/x) = p(x)$. If $p(x)$ is such an $\ell$-shifted reciprocal polynomial, we define $\tilde{p}(x) \coloneqq x^{-\ell/2} p(x)$.
\end{definition}

We observe that an $\ell$-shifted reciprocal polynomial $p(x)$ is reciprocal if and only if its constant coefficient does not vanish (in which case $\ell = \deg{p}$) and that $\tilde{p}(x)$ is a real-valued function on the unit circle $\T$, where $\tilde{p}(e^{\ii \theta})$ is interpreted as $e^{-\ii \theta \ell/2}p(e^{\ii \theta})$ for real $\theta$.

\begin{definition}
\label{def:zero type}
Suppose the discriminant $\Delta(P)$ is shifted reciprocal (i.e., it is an $\ell$-shifted reciprocal polynomial for some $\ell$) and vanishes at $\alpha = e^{\ii \theta} \in \T$ for some $\theta \in [0,2\pi]$. Then we say that $\alpha$ is of \emph{type} $\mr{I}$ if there is an $\epsilon > 0$ such that $\tilde{\Delta}(P)$ switches sign on and has exactly one zero in the set $\{e^{\ii t} : t \in (\theta-\epsilon, \theta+ \epsilon)\}$.
Similarly, we say that $\alpha$ is of \emph{type} $\mr{II}^+$ (or $\mr{II}^-$) if there is an $\epsilon>0$ such that $\tilde{\Delta}(P)$ is nonnegative (or nonpositive) on and has exactly one zero in the set $\{e^{\ii t} : t \in (\theta-\epsilon, \theta +\epsilon)\}$.
\end{definition}

This brings us to the definition of permissible polynomials.

\begin{definition}
\label{def:goodP}
Let $P = a(x)(y^2+1)+b(x)y \in \Z[x,y]$. Then $P$ is a \emph{permissible polynomial} if it  satisfies the following seven conditions:
\begin{enumerate}[label=(\arabic*)]
    \item There is an integer $\ell \geq 1$ such that $a$ and $b$ are $2\ell$-shifted reciprocal polynomials; \label{cond:reciprocal} 
    \item $a$ is the product of cyclotomic polynomials times a power of $x$; \label{cond:cyclotomic} 
    \item If $\deg{b} \ge \deg{a}$ then the leading coefficient of $b$ lies in $\{1,2\}$, and if $\deg{b} > \deg{a}$ then the leading coefficient of $b$ equals $1$; \label{cond:leading-coefficient} 
    \item $g$ and $h$ are coprime, and $a$ and $b$ are coprime; \label{cond:coprime}
    \item $\deg(g) \leq 4$; \label{cond:degree} 
    \item If $\deg(g) \ge 3$, then $h$ vanishes at $1$ or $-1$ and $\ord_{\alpha} h(x) \in \{0, 1\}$ for each $\alpha \in \T$; \label{cond:squarefree}
    \item For each $\alpha \in \T$ with $h(\alpha)=0$ that is of type $\mr{II}^+$ and for which $\ord_{\alpha} h(x)$ is odd, there exist a multiquadratic extension $K_{\alpha}/\Q$ (i.e., a finite Galois extension with elementary abelian $2$-group as Galois group) such that $W_\alpha \coloneqq \{\alpha, \sqrt{g(\alpha)}\} \subset K_{\alpha}$. \label{cond:W}
\end{enumerate}
Furthermore, suppose $P$ is permissible. Denote by $F_\alpha$ the set of fundamental discriminants of the
imaginary quadratic subfields of $K_\alpha$. Then we say that $P$ is
\emph{$(f_1,\ldots,f_s)$-permissible} if
\[
\{-f_1,\ldots,-f_s\}
=
\bigcup_{\substack{h(\alpha)=0\\ \alpha\in\T\ \text{of type }\mr{II}^+}}
F_\alpha.
\]
\end{definition}

\begin{conjecture}
\label{con:ourcon}
    Suppose that $P$ is an $(f_1, \ldots, f_s)$-permissible polynomial. Then $P$ is a Chinburg polynomial. More explicitly, there exist rational numbers $r_1, \ldots, r_s \in \Q$ such that
    \begin{equation*}
        \m(P) = r_1 L'(\chi_{-f_1},-1) + \cdots +  r_s L'(\chi_{-f_s}, -1).
    \end{equation*}
\end{conjecture}

\begin{remark}
\label{rem:independence}
In the stated generality, a proof of Conjecture~\ref{con:ourcon} may require the assumption of linear independence of the $L$-values $L'(\chi_{-f}, -1)$. However, in many of our examples, each $K_{\alpha}$ as in Condition~\ref{cond:W} may be taken to have only one pair of imaginary embeddings, in which case the assumption of linear independence of $L$-values is presumably not needed.
\end{remark}

\begin{remark}
\label{rem:nonzero}
It may happen that $r_1, \ldots, r_s$ are all zero. In that case $\m(P) = 0$ and by \cite[Theorem~1]{Boyd}, this implies that $P$ is a \emph{generalized cyclotomic polynomial}. An example is the polynomial $P$ obtained when $a = x^2$ and $b = x^4 + 1$.
\end{remark}

Next, we illustrate Conjecture~\ref{con:ourcon} with a couple of examples. In practice, the main challenge is to determine if a unimodular root $\alpha$ of $h$ is of type $\mr{II}^+$ or $\mr{II}^-$ (if $\Delta = gh^2$ with $g$ and $h$ coprime, then $\alpha$ cannot be of type $\mr{I}$ since it occurs with even multiplicity). The easiest way to do this is to compute all roots of $\Delta$ that lie on the unit circle, organise them according to their argument, and use the intermediate value theorem to determine if $\tilde{\Delta}$ is positive or negative in the vicinity of $\alpha$ on the unit circle.

\begin{example}[Conductor 43]
\label{ex:conductor-43}
Consider $a(x) = (1 + x)^2 (x^6 - x^5 + x^4 - x^3 + x^2 - x + 1)$ and $b(x) = 2x^8 + 2x^7 - 39x^6 - 16x^5 + 110x^4 - 16x^3 - 39x^2 + 2x + 2$ and set $P = a(x)(y^2+1) + b(x)y$.
We will verify that $P$ is a $43$-permissible polynomial. Conditions \ref{cond:reciprocal}-\ref{cond:leading-coefficient} in Definition \ref{def:goodP} can be checked easily. For the other conditions, we first compute that $\Delta(P) = gh^2$ with
\[
g = -(39 + 94x + 39x^2) \quad \text{ and } \quad h = x(x-1)(2+x-10x^2+x^3+2x^4). 
\]
This verifies Conditions \ref{cond:coprime}-\ref{cond:squarefree} in Definition \ref{def:goodP}. In this case, $\alpha = 1$ is the only root of $h$ that lies on the unit circle, and one may verify that it is of type $\mr{II}^+$. Since $g(1) = -4 \cdot 43$, we conclude that $P$ is indeed $43$-permissible. Assuming Conjecture~\ref{con:ourcon}, the number $\m(P)/L'(\chi_{-43},-1)$ is rational.
Using \texttt{lindep} from \texttt{PARI/GP} \cite{PARI}, we find that $\m(P)/L'(\chi_{-43},-1) \doteq 2/7$; that is, they are equal to at least $500$ decimal digits precision.
\end{example}

\begin{example}[Conductors $35$ and $3$]
Let
\begin{equation*}
 a(x) = (x^4 - x^3 + x^2 - x + 1)^2, \quad b(x) = 2x^8 - 4x^7 - 3x^6 + 16x^5 - 24x^4 + 16x^3 - 3x^2 - 4x + 2
\end{equation*}
and set $P = a(x)(y^2+1)+b(x)y$. Then $\Delta(P) = gh^2$ with
\[
g = -(2x^2 - 3x + 2) (2x^2 - x + 2) \quad \text{ and } \quad h = (x - 1) x (x + 1) (3x^2 - 4x + 3). 
\]
Thus Conditions \ref{cond:reciprocal}-\ref{cond:squarefree} certainly hold. Apart from $\pm 1$, the polynomial $h$ has the two additional roots $(2 \pm \sqrt{-5})/3$ that are of absolute value $1$. However, both of those are of type $\mr{II}^-$. Now $g(1)= -3$ and $g(-1) = -35$, and so we conclude that $P$ is $(3,35)$-permissible. Using \texttt{lindep} from \texttt{PARI/GP} \cite{PARI}, we find that 
\begin{equation*}
\m(P) \doteq r_1 L'(\chi_{-35},-1) + r_2 L'(\chi_{-3},-1)
\end{equation*}
holds to at least $500$ digits precision when taking $r_1 = 1/10$ and $r_2 = 7/5$.
\end{example}

\begin{example}[Conductor $3$]
If
\begin{equation*}
    a(x) = (x^2+1)^4 \quad \text{ and } \quad  b(x) = 2x^8 - 16x^7 - 8x^6 + 16x^5 + 44x^4 + 16x^3 - 8x^2 - 16x + 2
\end{equation*}
then 
\[
g = -x(x^2 + x + 1) \quad \text{ and } \quad h = 8(x - 1)(x + 1) (x^4 - 2x^3 - 2x^2 - 2x + 1).
\]
The polynomial $h$ has four roots on the unit circle, but only the root $\alpha=1$ has type $\mr{II}^+$. Since $g(1) = -3$, we find that $W_\alpha \subset \Q(\sqrt{-3})$ so $P$ is permissible. And indeed, \texttt{PARI/GP} yields $\m(P) \doteq 10 L'(\chi_{-3},-1)$.
\end{example}

\begin{example}[Ray's conductor 7 polynomial]
Ray \cite{Ray} considered the polynomial $P(x,y) = a(x)(y^2+1) + b(x)y$ where
\begin{equation*}
a(x) = x^6+x^5+x^4+x^3+x^2+x+1, \qquad b(x) = 2x^6 + 2 x^5 - 5 x^4 - 12 x^3 - 5 x^2 + 2 x + 2.
\end{equation*}
In this case, we obtain
\[
g = -7 \quad \text{ and } \quad h = (x - 1) x (x + 1) (2x^2 + 3x + 2).
\]
The roots of $h$ on the unit circle are $\pm 1, \frac{-3 \pm \sqrt{-7}}{4}$;
these lie in $\Q(\sqrt{-7})$ and are all of type $\mr{II}^+$. (This is a consequence of the identity $h(x) = -x^6 h(1/x)$, so that $\tilde{h}(x)$ is purely imaginary on the unit circle, and that $g = \tilde{g}$ is negative.) Since $\sqrt{g} = \sqrt{-7} \in \Q(\sqrt{-7})$ as well, we conclude that $P$ is permissible. Ray proved the equality $\m(P) = \frac{8}{7} L'(\chi_{-7}, -1)$.
\end{example}

\begin{example}
Let us show why Condition~\ref{cond:squarefree} is supposedly important. Consider the case
\begin{equation*}
a(x) = x^2, \qquad b(x) = 1 + 4 x + 4 x^2 + 4 x^3 + x^4
\end{equation*}
so that
\begin{equation*}
g(x) = (1 + x^2)  (1 + 4 x + x^2), \qquad h=(x+1)^2;
\end{equation*}
this is not a permissible polynomial since $\deg g(x) \geq 3$ and $\ord_{-1}h(x) \not \in \{0,1\}$, but all other conditions of Definition~\ref{def:goodP} are satisfied. One may verify that the root $-1$ is of type $\mr{II}^-$, so that $F_{-1}$ is empty. However, the Mahler measure of $P$ is certainly nonzero.  In fact, we conjecture that $\m(a(x)(y^2+1) + b(x)y) = 2 L'(E_{24a1},0)$, where $E_{24a1}$ is the conductor $24$ elliptic curve, corresponding to the projective closure of the desingularization of $w^2 = g(x)$.
Another, similar example comes from the choice $a(x) = x^3 + x$ and $b(x) = 1 + 2 x + 6 x^2 + 2 x^3 + x^4$, in which case we conjecture $\m(a(x)(y^2+1) + b(x)) = 2L'(E_{32a1},0)$, where the conductor $32$ elliptic curve $E_{32a1}$ arises in a similar way. 
\end{example}

\begin{example}
Another example we cannot explain is $a(x) = (x^2 + x + 1)(x^4 - x^2 + 1)$ and $b(x) = (x - 1)^2 (2x^4 + 6x^3 - 17x^2 + 6x + 2)$, in which case $g(x) =  -3 (9x^2 - 16x + 9)(x^2 + 4x + 1)$ and $h(x) = x(2x^2+3x+2)$.  Here, $h$ does not vanish at $1$ or $-1$, although all other conditions of Definition~\ref{def:goodP} are satisfied. A computation similar to those in the examples above suggests a linear relation involving $L'(\chi_{-1155},-1)$ and $L'(\chi_{-7},-1)$, and perhaps $L'(E_{15912h1},0)$, where $E_{15912h1}$ is the elliptic curve of conductor $15912$ arising from the curve $w^2=g(x)$. However, no such relation was found with \texttt{PARI/GP}.
\end{example}

\begin{example}
\label{ex:BRV1}
Almost all examples in \cite{BRV1} can be put in the context of permissible polynomials, and their results can be recovered from our Conjecture~\ref{con:ourcon}.
The polynomials considered in \cite{BRV1} are of the form $p(x)y+q(x)$,
where $p$ and $q$ are products of cyclotomic polynomials times a power of $x$,
with $\gcd(p(x),q(x))=1$.
Suppose for simplicity that $p(x)$ and $q(x)$ are $\ell$-shifted reciprocal. (Often this can be assumed after some monomial transformation that preserves the Mahler measure: for example, $\m(1+x+y) = \m(1+x^2+y) = \m(1+x^2+xy)$, see, e.g., \cite[p.~34]{BZ}.) Note that
\[
 \m(p(x)y+q(x))
 =
 \m(q(x)y+p(x))
 =
 \frac12 \m(P),
\]
where $P\coloneqq a(x)(y^2+1)+b(x)y$ and
\[
 a(x)=p(x)q(x) \qquad \text{and} \qquad b(x)=p(x)^2+q(x)^2.
\]
This implies immediately that
Conditions~\ref{cond:reciprocal}, \ref{cond:cyclotomic} and \ref{cond:leading-coefficient} of Definition \ref{def:goodP} are satisfied for these $a(x)$ and $b(x)$. Furthermore, since
\[
 \Delta(P)(x)=(p(x)^2-q(x)^2)^2
\]
we find $g\equiv1$ and $h=p(x)^2-q(x)^2$, so Conditions~\ref{cond:coprime}, \ref{cond:degree} and \ref{cond:squarefree} are satisfied.

Moreover, $h(\alpha)=0$ for $\alpha\in\T$ if and only if $\alpha^\ell(p(\alpha)p(1/\alpha)-q(\alpha)q(1/\alpha))=0$, which is equivalent to $|p(\alpha)|=|q(\alpha)|$. Since $\tilde h(\alpha)$ is real-valued on the torus, $\tilde\Delta(P)(\alpha)$ is nonnegative, so any $\alpha$ is of type $\mr{II}^+$. Furthermore, the toric roots of $h(x)$ of odd order in Condition \ref{cond:W} are precisely the roots of what is called the \emph{crossing polynomial} in \cite{BRV1}. 
\end{example}

\begin{remark}
\label{rem:conjecture}
In this paper, we focus on generating examples, but here we briefly outline what we believe to be the main steps of a possible proof of
Conjecture~\ref{con:ourcon} following the approach of Boyd and
Rodriguez-Villegas in \cite{BRV2}. The conditions in
Definition~\ref{def:goodP} are of two types: geometric and algebraic.

Conditions \ref{cond:reciprocal}-\ref{cond:degree} are of geometric
origin. Let $Y$ be the zero locus of $P$ and let $X$ be its smooth
projective completion. Then Conditions \ref{cond:coprime} and
\ref{cond:degree} imply that $X$ has genus $0$ or $1$. Moreover,
Conditions \ref{cond:reciprocal}, \ref{cond:cyclotomic}, and
\ref{cond:leading-coefficient} imply that the symbol
$\{x,y\}$ lies in $K_2(X)\otimes\Q$, where $K_2(X)$ denotes the $K_2$-group of $X$. If the genus of $X$ is $0$, this implies
that $x \wedge y$ admits a \emph{triangulation} (see \cite[p.~3]{BRV2}). In the case that $X$ has genus $1$, we expect that the additional
Condition \ref{cond:squarefree} implies that $x \wedge y$ admits a
triangulation in that case as well. Indeed, the map
$\sigma \colon Y \to Y$ given by $(x,y) \mapsto (1/x,1/y)$ has fixed points, and the quotient $X/\langle \sigma\rangle$ has genus $0$.

Let $e_k \in X$ be a boundary point of the \emph{Deninger path}
(see \cite[p.~3]{BRV2}). We believe that the triangulations obtained
above imply the existence of an element $
\xi_k \in B(\overline{\Q}) \otimes \Q$,
in the \emph{Bloch group}. The Mahler measure can then be written as a
sum of the values $D(\xi_k)$, where $D$ denotes the
\emph{Bloch--Wigner dilogarithm}, taken over all boundary points $e_k$.

The final Condition \ref{cond:W} is more algebraic in nature, and we
expect that it guarantees $\xi_k \in B(L_k) \otimes \Q$, where $L_k$
is a multiquadratic field. Writing
$\{-f_{1,k}, \ldots, -f_{s,k}\}$ for the fundamental discriminants of
the quadratic subfields of $L_k$, and assuming the linear independence
of $L$-values, the group $B(L_k) \otimes \Q$ is generated over $\Q$ by
the values $L'(\chi_{-f_{j,k}}, -1)$. (In many of our examples, the fields $L_k$ can be taken to be quadratic imaginary, so no linear independence assumption is needed.) Summing over all boundary points $e_k$, which we believe correspond precisely to the zeros of $h(x)$ of odd order that have type $\mathrm{II}^+$, this would imply
Conjecture~\ref{con:ourcon}. It may be possible to derive this last step from the work of Guilloux--Marché \cite{GM}.
\end{remark}

\section{Algorithm and discussion of tables of new and old results}
\label{sec:tables}

In this section, we provide an overview of previously known and new results on Conjecture~\ref{con:Chinburg}, and discuss the algorithm that we used to obtain these. All tables with results are provided in the appendix.
Table~\ref{tab:previous_strong} provides an overview of examples of the strong form of Conjecture~\ref{con:Chinburg} that we could find in the literature; we only list one example per conductor.
Table~\ref{tab:new_strong} provides an overview of permissible polynomials that we found in this work giving examples of the strong form of Conjecture~\ref{con:Chinburg}, assuming Conjecture~\ref{con:ourcon}.
Lastly, Table~\ref{tab:T1} gives an overview of all other permissible polynomials that we found with their Mahler measures, which in all cases lead to examples of the weak form of Conjecture~\ref{con:Chinburg} assuming Conjecture~\ref{con:ourcon}. In Tables~\ref{tab:new_strong} and \ref{tab:T1}, we also give only one example per conductor or conductor tuple.
The claims in these tables are all of the form $\m(P) = r_{1} d_{f_1} + \cdots + r_{s} d_{f_s}$ for some rational, unknown numbers $r_1, \ldots, r_s$ and explicit conductors $f_1, \ldots, f_s$. In each case, we used \cite{LQ-code} for the numerical computation of $\m(P)$ and the \texttt{lindep} function of \texttt{PARI/GP} \cite{PARI} to determine the $r_j$ up to at least $500$ decimal digits.

\subsection*{Algorithm}

Our algorithm, and in fact the definition of permissibility, is motivated by the general structures appearing in
Table~\ref{tab:previous_strong}. Suppose $P$ is permissible. First, by Condition~\ref{cond:degree} we have
\[
\Delta(P) = (b(x)-2a(x))(b(x)+2a(x)) = g(x)h(x)^2,
\]
where $g$ has degree at most $4$ and $a(x)$ and $b(x)$ are as in Conditions~\ref{cond:reciprocal}, \ref{cond:cyclotomic} and \ref{cond:coprime}.
Moreover $\gcd(g(x),h(x))=1$ and $g(x)$ is squarefree.

Examining the examples in Table~\ref{tab:previous_strong} that are quadratics in $y$, we observe that in each case $h(1)=0$ or $h(-1)=0$. Furthermore,
in these examples one of the factors $b(x)-2a(x)$ or $b(x)+2a(x)$ is
of the form $c\,v(x)^2$, where $c\in\Z$ is a constant and $v(x)$ is an
$\ell$-shifted reciprocal polynomial, usually of small height.
The only exceptions in Table~\ref{tab:previous_strong} occur for the
polynomials corresponding to conductors $303$ and $755$. In those cases one finds a factorization of the form
\[
b(x)+2a(x) = v(x)^2 q(x),
\]
where $q(x)$ is a monic quadratic polynomial. We did not find new examples of the latter shape.

Motivated by these observations, we impose the following conditions on
our search space:
\begin{enumerate}[label=(\roman*)]
\item One of the factors $b(x)-2a(x)$ or $b(x)+2a(x)$ is of the form
\[
b(x)\pm 2a(x) = c v(x)^2,
\]
where $c\in\Z$ is nonzero and squarefree and $v(x)$ is an
$\ell$-shifted reciprocal polynomial. \label{cond: v}
\item Any $\alpha$ of type $\mathrm{II}^{+}$ has $[\Q(\alpha):\Q]\le 2$, and
and $\Q(\sqrt{g(\alpha)})$ is contained in a biquadratic extension of $\Q$ (equivalently, it is Galois with Galois group isomorphic to a subgroup of the Klein four-group).
\label{cond: Galois}
\item $h$ vanishes at $-1$ or $1$. \label{cond: h_vanish}
\end{enumerate}

Note that Condition~\ref{cond: Galois} immediately implies Condition~\ref{cond:W} in Definition~\ref{def:goodP}, but it is strictly stronger. For example, one could additionally allow the case where $\Q(\alpha)/\Q$ is biquadratic and
\[
\Gal(\Q(\sqrt{g(\alpha)})/\Q) \cong \Z/2 \times \Z/2 \times \Z/2.
\]
We chose not to incorporate this more general condition in our algorithm.

For $\ell \geq 2$, write
\[
v(x)=a_0+a_1x+\cdots+a_1x^{\ell-1}+a_0x^\ell .
\]
In the algorithm, we assume that the coefficients $a_0,\dots,a_{\lfloor (\ell - 1)/2 \rfloor}$ and the constant $c$ are bounded above in absolute value by a fixed constant $B$. If $\ell$ is even, the coefficient $a_{\ell/2}$ is determined by $h(\pm 1) = 0$ and the other coefficients $a_j$ in Condition \ref{cond: h_vanish}.
By replacing $v$ by $-v$ if necessary we may assume $a_0\ge0$. 
Condition~\ref{cond:cyclotomic} and Condition~\ref{cond:leading-coefficient} necessitate $a(0)\in\{0,\pm1\}$ and $b(0)\in\{0,1,2\}$, respectively, depending on whether or not $\deg b>\deg a$. Again, both cases are explored:
\begin{itemize}
    \item \textbf{Case 1:}  $a(0)\in\{\pm1\}$. Then $b(0)\in\{0,1,2\}$ and the relation $b\pm 2a = c v^2$ implies $c a_0^2 \in \{-2,-1,0,2,3,4\}$, so that
    \[
        (a_0,c)\in
        \{(1, -1),(1,\pm 2),(1,3),(2,1)\}
        \cup
        \{(0,c) : 0<|c|\le B,\; c\ \text{squarefree}\}.
    \]
\item \textbf{Case 2:} $a(0)=0$. Since $\gcd(a,b)=1$, we must have $b(0)=1$.  
In this case the relation $b\pm2a=cv^2$ gives $c a_0^2 = 1$, so that
\[
(a_0,c)=(1,1).
\]
\end{itemize}

We loop over all $a(x)$ as in Conditions~\ref{cond:reciprocal} and \ref{cond:cyclotomic} with $a_0$ and $c$ as above, and with 
$a_1,\dots,a_{\lfloor (\ell - 1)/2 \rfloor}$ bounded in absolute value by $B$, subject to Conditions~\ref{cond: v}, \ref{cond:squarefree} and \ref{cond: h_vanish}. 
For each resulting pair $(a,b)$, we verify Definition~\ref{def:goodP}, replacing Condition~\ref{cond:W} with the stronger Condition~\ref{cond: Galois}, and return the corresponding conductors $f_1,\ldots,f_s$.

In Table~\ref{tab:output} we list the number of distinct outputs of our algorithm 
for each $2 \le \ell \le 7$, where \emph{strong} means $s=1$ and \emph{weak} means $s>1$. 
The corresponding polynomials are collected in Tables~\ref{tab:new_strong} and~\ref{tab:T1}. 
For each tuple $(f_1,\ldots,f_s)$ produced by our algorithm, we record only one polynomial 
that is $(f_1,\ldots,f_s)$-permissible.  
The algorithm and the complete output of the algorithm when run as in Table~\ref{tab:output} can be found at
\begin{center} 
\url{https://github.com/davidhokken/chinburg}. 
\end{center} 

\section{Chinburg's weak conjecture for rational functions with cyclotomic coefficients}
\label{sec:algebraicChinburg}
Following Chinburg's original methods \cite{Chinburg}, we can prove Chinburg's weak conjecture for the conductor $f$ if we allow $P \in \Q(\zeta_{4f})(x,y)$. This is a consequence of the following theorem.

\begin{theorem}
\label{thm:algchin}
Let $\chi=\chi_{-f}$ be an odd quadratic character of conductor $f$, and put
\[
\alpha_f \coloneqq
\prod_{0<u<\frac{f}{2}}
\left|2\sin\!\left(\frac{\pi u}{f}\right)\right|^{- 2 \chi(u) u}.
\]
Then
\begin{equation}
\label{eq:algebraicChinburg}
\m\!\Bigg(
\alpha_f
\prod_{0<u<\frac{f}{2}}
\Big(x+y+2\sin\!\frac{\pi u}{f}\Big)^{\chi(u)f}
\Bigg)
=
2d_f.
\end{equation}
\end{theorem}
\begin{remark}
Note that
\[
\alpha_f
\prod_{0<u<\frac{f}{2}}
\left(x+y+2\sin\!\left(\frac{\pi u}{f}\right)\right)^{\chi(u)f} \in \Q(\zeta_{4f})(x,y).
\]
Thus Theorem~\ref{thm:algchin} gives a positive answer to Chinburg's weak conjecture when one allows rational functions with cyclotomic coefficients.  
\end{remark}
\begin{proof}
We use the formula of Cassaigne--Maillot \cite[Proposition~7.3.1]{Ma}: If
$a,b,c\in\C$ are such that $|a|,|b|,|c|$ are the side lengths of a
non-degenerate triangle with opposite angles $\alpha,\beta,\gamma$, then
\[
\m(ax+by+c)
=
\frac{1}{\pi}
\left(
\alpha\log|a|+\beta\log|b|+\gamma\log|c|
+D\!\left(\frac{|b|}{|a|}e^{i\gamma}\right)
\right),
\]
where $D(z)$ denotes the Bloch--Wigner dilogarithm. 
Now, for $0<u<f/2$, set
\[
G(u) \coloneqq \m\!\left(x+y+2\sin\!\left(\frac{\pi u}{f}\right)\right).
\]
Applying the Cassaigne--Maillot formula to the nondegenerate, isosceles triangle with side
lengths $|a|=|b|=1$ and $|c|=|2\sin(\pi u/f)|$, whose opposite angles
are $\alpha=\beta=\pi/2-\pi u/f$ and $\gamma=2\pi u/f$, we obtain
\begin{equation*}
G(u)=\frac{2u}{f}\log\left|2\sin\!\left(\frac{\pi u}{f}\right)\right|
+\frac{1}{\pi}D\!\left(e^{2\pi i u/f}\right). 
\end{equation*}
Hence by additivity of the Mahler measure, the left-hand side of \eqref{eq:algebraicChinburg} equals
\begin{equation}
\label{eq:AC2}
\log(\alpha_f) + f\sum_{0<u<\frac{f}{2}}\chi(u)G(u) = \frac{f}{\pi}\sum_{0<u<\frac{f}{2}}\chi(u)\, D\!\left(e^{2\pi i u/f}\right) 
\end{equation}
where we used the definition of $\alpha_f$ to cancel terms.
On the other hand, using $D(\bar z)=-D(z)$ and $\chi(-u)=-\chi(u)$, we have
\begin{equation}
\label{eq:AC3}
2d_f
=
\frac{f}{2\pi}\sum_{0<u<f}\chi(u)\,D\!\left(e^{2\pi i u/f}\right) = 
\frac{f}{\pi}\sum_{0<u<\frac{f}{2}}\chi(u)\,
D\!\left(e^{2\pi i u/f}\right)
\end{equation}
where the first step is an identity given in \cite[p.~74]{Boyd-L}.
Combining \eqref{eq:AC2} and \eqref{eq:AC3} concludes the proof.
\end{proof}

\newpage
\appendix
\section{Tables}

\begin{table}[ht!] 
\centering
\renewcommand{\arraystretch}{1.0}
\tiny{
\begin{tabular}{llll} 
\toprule 
$\m(P)$ & $P(x,y)$ & Source \\
\midrule
$d_3$ (!)& $ 1+x+y$ & \cite{Smyth} \\
$d_4$ (!)& $(1-x)y+(1+x)$ &  \cite{Boyd-L} \\
$\frac{8}{7} d_7$ (!)& $ (x^6+x^5+x^4+x^3+x^2+x+1) (y^2+1) + (2x^6 + 2 x^5 - 5 x^4 - 12 x^3 - 5 x^2 + 2 x + 2) y$ & \cite{Ray} \\
$d_8$ (!)& $ (x^4+1)(y^2+1)+(2x^4-8x^2+2)y$ & \cite{Ray} \\
$\frac{2}{3} d_{11}$ & $ (x+1)^2(x^2+x+1)y-(x^2-x+1)^2$ & \cite{BRV1} \\
$6d_{15}$ & $ (x+1)^2y-(x^2-x+1)$ & \cite{BRV1} \\
$\frac{2}{5} d_{19}$& $ (x^5+x^4+x^3+x^2+x)(y^2+1)+(x^6+6x^5+2x^4-8x^3+2x^2+6x+1)y$ &  \cite{BRV2} \\
$\frac{2}{5} d_{20}$ (!)& $ (x^8-x^6+x^4-x^2+1)(y^2+1)+(2x^8 - 22 x^6 + 42 x^4 - 22 x^2 + 2)y$ & \cite{Ray} \\
$\frac{1}{6}d_{23}$& $ (x^5-x^3+x) (y^2+1)+(x^6-6x^5+12x^3-6x+1)y$ & \cite{LQ} \\
$\frac{1}{3}d_{24}$ (!)& $ (x^8-x^4+1)(y^2+1)+(2x^8 - 24 x^6 + 46 x^4 - 24 x^2 + 2)y$ & \cite{Ray} \\
$\frac{1}{12} d_{35}$ & $ (x^2+x+1)^2y-(x^2+1)(x^2-x+1)$ & \cite{BRV1} \\
$\frac{1}{18}d_{39}$& $ (x+1)^4y-(x^2+x+1)(x^2-x+1)$ & \cite{BRV1} \\
$\frac{1}{6}d_{40}$&  $(x^4 - x^3 + 2 x^2 - x + 1)(y^2+1)+x(14x^2-32x+14)y$ & \cite{BRV2} \\
$\frac{1}{30}d_{55}$& $ (x+1)^2(x^2+x+1)y-(x^4-x^3+x^2-x+1)$ & \cite{BRV1} \\
$\frac{1}{36}d_{84}$& $ (x+1)^4y-(x^2+1)(x^2-x+1)$ & \cite{BRV1} \\
$\frac{1}{36}d_{120}$& {$(x^2+1)(x^2+x+1)(y^2+1)+2(x^2-3x+1)(x^2+4x+1)y$} & \cite{BRV2} \\
$\frac{1}{132}d_{303}$ & ${(x^8+x^7+x^6+x^2+x+1)(y^2+1)+(2x^8+2x^7-49x^6+2x^5+98x^4+2x^3-49x^2+2x+2)y}$ &  \cite{LQ} \\
$\frac{1}{410}d_{755}$ & ${(x^8+x^6+x^4+x^2+1)(y^2+1)+(2x^8-37x^6+5x^5+70x^4+5x^3-37x^2+2)y}$ & \cite{LQ} \\
\bottomrule
\end{tabular}}
\caption{Overview of examples of the strong form of Conjecture~\ref{con:Chinburg} known prior to this work; we only list one example per conductor. The second column contains polynomials $P(x,y)$, with their Mahler measures $\m(P)$ listed in the first column. In each example, it is known that $r_f \coloneqq \m(P)/d_f \in \Q^{\times}$, but the listed values of $r_f$ are inferred from numerical computation only, except for the proven ones marked with (!). All polynomials linear in $y$ except the one for $d_4$ (as $x - 1$ is not shifted reciprocal) can be made into a permissible polynomial by Example \ref{ex:BRV1}.}
\label{tab:previous_strong}
\end{table}

\begin{table}[ht!] 
\centering
\renewcommand{\arraystretch}{1.2}
\tiny
\begin{tabular}{ll} 
\toprule 
$\m(P)$ & $a(x); \ b(x)$ \\
\midrule
$\frac{2}{7}d_{43}$ & $x^8 + x^7 + x + 1$ \\ 
& $2x^8 + 2x^7 - 39x^6 - 16x^5 + 110x^4 - 16x^3 - 39x^2 + 2x + 2$ \\
$\frac{2}{21}d_{52}$ & $x^8 - 2x^7 + 3x^6 - 3x^5 + 3x^4 - 3x^3 + 3x^2 - 2x + 1$ \\
                     &$2x^8 - 4x^7 - 10x^6 + 6x^5 + 14x^4 + 6x^3 - 10x^2 - 4x + 2$ \\
$\frac{1}{15}d_{56}$ & $x^8 - x^7 + x^4 - x + 1$ \\ 
& $2x^8 - 2x^7 - 15x^6 + 4x^5 + 24x^4 + 4x^3 - 15x^2 - 2x + 2$ \\
$\frac{1}{18}d_{68}$ & $x^8 - x^7 + x^6 - x^5 + x^4 - x^3 + x^2 - x +1$ \\ 
& $2x^8 - 2x^7 - 17x^6 + 6x^5 + 24x^4 + 6x^3 - 17x^2 - 2x + 2$ \\
$\frac{1}{54}d_{111}$ & $x^8 - x^7 + x^6 + x^5 - x^4 + x^3 + x^2 - x + 1$ \\
& $x^8 + 8x^7 - 11x^6 - 8x^5 + 26x^4 - 8x^3 - 11x^2 + 8x + 1$ \\
$\frac{1}{30}d_{132}$ & $x^8 + x^5 + x^3 + 1$ \\
& $2x^8 - 32x^6 - 2x^5 + 72x^4 - 2x^3 - 32x^2 + 2$ \\
$\frac{1}{54}d_{228}$ & 
$x^8 + 2x^7 + x^6 - x^5 - 2x^4 - x^3 + x^2 + 2x + 1$ \\
& $2x^8 + 4x^7 - 46x^6 - 38x^5 + 164x^4 - 38x^3 - 46x^2 + 4x + 2$ \\
$\frac{1}{252} d_{696}$ & $x^{12} + 2x^{10} + 2x^8 + 2x^6 + 2x^4 + 2x^2 + 1$ \\& $x^{12} - 54x^{11} + 245x^{10} + 54x^9 - 967x^8 + 1466x^6 - 967x^4 + 54x^3 + 245x^2 - 54x + 1$ \\
\bottomrule
\end{tabular}
\caption{Overview of permissible polynomials giving rise to examples of the strong form of Conjecture~\ref{con:Chinburg} new in this work, assuming Conjecture~\ref{con:ourcon}. The second column gives two polynomials $a(x)$ and $b(x)$ such that the polynomial $P(x,y) = a(x)(y^2+1)+b(x)y$ is permissible. The first column then gives the corresponding numerically inferred Mahler measure $\m(P)$.}
\label{tab:new_strong}
\end{table}
\clearpage

{\tiny
\renewcommand{\arraystretch}{0.9}
\begin{longtable}{l  p{0.6\textwidth}}

\toprule
$\m(P)$ & $a(x); \ b(x)$ \\
\midrule
\endfirsthead

\toprule
$\m(P)$ & $a(x); \ b(x)$ \\
\midrule
\endhead

\bottomrule
\endfoot

$\frac{8}{5}(-d_{4} + 5d_{3})$  & $x^4 - x^3 + x^2 - x + 1; 2 x^4 - 6 x^3 + 10 x^2 - 6 x + 2$ \\

$\frac{1}{3} (3 d_7 + 5 d_3)$  & $x^6 - x^5 + 3x^4 - 2x^3 + 3x^2 - x + 1; 2x^6 - 2x^5 + 2x^4 - 12x^3 + 2x^2 - 2x + 2$ \\

$ \frac{4}{15}(7d_{7} - 4d_{4})$  & $x^8 + 3x^7 + 3x^6 + x^5 + x^3 + 3x^2 + 3x + 1$\\& $ x^8 + 4x^7 - 3x^6 - 8x^5 - 20x^4 - 8x^3 - 3x^2 + 4x + 1$\\

$\frac{1}{2}(-d_{8} + 27d_{3})$ & $x^5 + x; x^6 + 10x^5 + 24x^4 + 34x^3 + 24x^2 + 10x + 1$\\

$ \frac{4}{15}(d_8 + d_4) $  & $x^4 - x^3 + x^2 - x + 1; x^4 - 2x^3 - 2x + 1$\\

$ \frac{1}{3} (-5d_8 + 12 d_7)$  &  $x^6 + 1 ; 2x^6 + 32 x^5 + 96 x^4 + 136 x^3 + 96 x^2 + 32 x + 2$ \\

$\frac{2}{3}(2 d_{11} + d_{3})$  & $x^6 + x^5 + 2x^4 + x^3 + 2x^2 + x + 1; 2x^6 - 22x^5 + x^4 + 56x^3 + x^2 - 22x + 2$ \\

$ \frac{2}{5} (d_{11} + 4 d_4)$  & $x^6 - x^5 - x + 1; 2x^6 - 2x^5 - 3x^4 + 10x^3 - 3x^2 - 2x + 2$\\

$\frac{1}{6} (3d_{15} + 10 d_3)$  & $x^6 - x^5 + 3x^4 - 2x^3 + 3x^2 - x + 1; 2x^6 - 14x^5 - 6x^4 + 44x^3 - 6x^2 - 14x + 2$\\

$\frac{1}{30}(d_{15} + 26 d_7)$  & $x^6 + x^4 + x^3 + x^2 + 1 ; 2x^6 - 3x^4 - 8x^3 - 3x^2 + 2$\\

$ \frac{1}{21} (d_{15} + 4 d_{7} + 8 d_4)$  & $x^6 - x^5 + x^4 - x^3 + x^2 - x + 1; x^6 - 2x^5 + 2x^4 - 4x^3 + 2x^2 - 2x + 1$\\

$ \frac{1}{30} (5d_{15} - 2 d_8)$  & $x^6 + x^5 + 2x^4 + 2x^3 + 2x^2 + x + 1 ; x^6 + 2x^5 + 4x^4 + 2x^3 + 4x^2 + 2x + 1$\\

$\frac{5}{9}(d_{19} + d_3)$  & $x^6 + x^3 + 1;  x^6 - 12x^5 + 28x^3 - 12x + 1$\\

$\frac{1}{4}(d_{20} + 4d_{4})$   & $x^6 + 3x^4 + 3x^2 + 1; 2x^6 - 2x^4 - 16x^3 - 2x^2 + 2$\\

$ \frac{2}{15}(d_{20} + 4 d_8)$   & $x^6 + x^4 - x^3 + x^2 + 1; 2x^6 - 2x^4 + 10x^3 - 2x^2 + 2$\\

$\frac{1}{3}(d_{20} + d_{11})$  & $x^{14} - x^{13} + 4x^{12} - 3x^{11} + 6x^{10} - 3x^9 + 5x^8 - 2x^7 + 5x^6 - 3x^5 + 6x^4 - 3x^3 + 4x^2 - x + 1$\\ &$2x^{14} - 2x^{13} - 7x^{12} - 14x^{11} - 22x^{10} + 16x^9 + 11x^8 + 64x^7 + 11x^6 + 16x^5$ \\ & $\qquad- 22x^4 - 14x^3 - 7x^2 - 2x + 2$\\

$\frac{2}{15}(d_{20} +3d_{15} )$  & $x^{12} + 3x^{11} + 5x^{10} + 4x^9 - 5x^7 - 7x^6 - 5x^5 + 4x^3 + 5x^2 + 3x + 1 $\\& $ 2x^{12} + 6x^{11} + 15x^{10} + 48x^9 + 80x^8 + 110x^7 + 136x^6 + 110x^5 + 80x^4 + 48x^3$ \\& \qquad $+ 15x^2 + 6x + 2$\\

$\frac{1}{14}(d_{23} + d_7)$ & $ x^6 - x^5 + x^4 - x^3 + x^2 - x + 1 ; 2x^6 - 2x^5 + x^4 - 4x^3 + x^2 - 2x + 2$\\

$\frac{2}{15}(d_{23} + d_{15})$ & $x^6 + 2x^5 + 3x^4 + 3x^3 + 3x^2 + 2x + 1$; $x^6 + 6x^4 + 16x^3 + 6x^2 + 1$ \\

$\frac{1}{2}(-d_{24} + 12d_7)$ & $x^8 - 4x^7 + 6x^6 - 4x^5 + 2x^4 - 4x^3 + 6x^2 - 4x + 1$\\&$  2x^8 - 8x^7 - 188x^6 - 488x^5 - 684x^4 - 488x^3 - 188x^2 - 8x + 2$\\

$\frac{1}{3}(d_{24} - 2d_{8})$ & $x^4 - x^2 + 1; 2x^4 - 8x^3 + 14x^2 - 8x + 2$\\
 
$\frac{1}{9}(4d_{24} - 20d_{8} - 28d_4 +120d_3)$ & $x^8 - 2x^7 + x^6 + x^5 - 2x^4 + x^3 + x^2 - 2x + 1 $\\& $ 2x^8 + 12x^7 + 18x^6 + 22x^5 + 36x^4 + 22x^3 + 18x^2 + 12x + 2$\\
 
$\frac{1}{5}(2d_{24} - d_{20})$ & $x^6 - 3x^5 + 4x^4 - 4x^3 + 4x^2 - 3x + 1; 2x^6 - 6x^5 + 4x^4 - 16x^3 + 4x^2 - 6x + 2$ \\
 
$\frac{2}{21}(d_{31} + 4d_7)$ & $x^6 - x^5 + x^4 - x^3 + x^2 - x + 1$ ; $x^6 + 2x^5 + 2x^4 - 12x^3 + 2x^2 + 2x + 1$ \\

$\frac{2}{15}(d_{31} + d_{15})$ & $x^6 - 2x^5 + 3x^4 - 3x^3 + 3x^2 - 2x + 1$ ; $2x^6 - 8x^5 - 9x^4 + 32x^3 - 9x^2 - 8x + 2$ \\

$\frac{1}{6}(d_{35} + 6d_3)$ & $x^6 - x^5 + x^3 - x + 1; 2x^6 - 10x^5 - 3x^4 + 24x^3 - 3x^2 - 10x + 2$\\

$\frac{1}{30}(5 d_{35} + 3d_{15})$ & $ x^{10} - 2x^9 + 5x^8 - 7x^7 + 10x^6 - 10x^5 + 10x^4 - 7x^3 + 5x^2 - 2x + 1 $\\& $ 2x^{10} - 4x^9 - 14x^8 + 14x^7 - 16x^6 + 44x^5 - 16x^4 + 14x^3 - 14x^2 - 4x + 2$\\

$\frac{1}{70}(-d_{35} + 10d_{15} + 114d_{3})$ &
$x^{10} - 2x^9 + 3x^8 - 4x^7 + 5x^6 - 5x^5 + 5x^4 - 4x^3 + 3x^2 - 2x + 1 $\\& $ 2x^{10} - 4x^9 + 5x^8 - 8x^7 + 10x^6 - 12x^5 + 10x^4 - 8x^3 + 5x^2 - 4x + 2$\\

$\frac{1}{12} (d_{39} + 20d_3)$ & $x^6 + x^5 + 3x^4 + 2x^3 + 3x^2 + x + 1; 2x^6 + 2x^5 - 6x^4 - 20x^3 - 6x^2 + 2x + 2$\\

$ \frac{1}{21}(d_{39} + 18 d_4)$ & $x^6 + x^5 + x^4 + x^3 + x^2 + x + 1; 2x^6 + 2x^5 + 3x^4 + 8x^3 + 3x^2 + 2x + 2$\\

$\frac{2}{21}(d_{39} + 3d_7)$ & $ x^6 - x^5 + x^4 - x^3 + x^2 - x + 1; 2x^6 - 10x^5 - 5x^4 + 28x^3 - 5x^2 - 10x + 2$\\

$\frac{1}{45}(4 d_{39} + 3d_{15})$ & $x^6 + 2x^5 + 3x^4 + 3x^3 + 3x^2 + 2x + 1; 2x^6 + 4x^5 - 9x^4 - 24x^3 - 9x^2 + 4x + 2$\\

$\frac{1}{90}(-d_{39} + 12d_{24} + 12d_{15})$ &$x^6 + x^4 - x^3 + x^2 + 1; x^6 - 2x^4 + 8x^3 - 2x^2 + 1$\\

$\frac{1}{10}(d_{40} + 2d_{3})$ & $ x^6 - x^5 + 2x^4 - 2x^3 + 2x^2 - x + 1; 2x^6 - 2x^5 - 7x^4 + 14x^3 - 7x^2 - 2x + 2$\\

$ \frac{4}{15}(d_{40} -11d_{4} )$ & $x^8 - x^5 - x^3 + 1 $\\& $ 2x^8 + 16x^7 + 48x^6 + 82x^5 + 104x^4 + 82x^3 + 48x^2 + 16x + 2$\\

$ \frac{2}{21}(2d_{40} - 3d_{7})$ & $x^{10} - 3x^9 + 6x^8 - 8x^7 + 9x^6 - 9x^5 + 9x^4 - 8x^3 + 6x^2 - 3x + 1$\\& $ x^{10} - 8x^9 + 17x^8 - 4x^7 + 13x^6 - 40x^5 + 13x^4 - 4x^3 + 17x^2 - 8x + 1$\\

$\frac{1}{10} (d_{40} + 6d_8)$ & $x^8 + 2x^7 + 3x^6 + 4x^5 + 5x^4 + 4x^3 + 3x^2 + 2x + 1 $\\& $ 2x^8 + 4x^7 + 2x^6 - 16x^5 - 34x^4 - 16x^3 + 2x^2 + 4x + 2$\\

$ \frac{4}{15}(d_{40} - d_{20}) $ & $x^8 + x^7 - x^5 - x^4 - x^3 + x + 1 $\\& $ 2x^8 - 18x^7 + 40x^6 - 62x^5 + 78x^4 - 62x^3 + 40x^2 - 18x + 2$\\

$\frac{2}{15}(d_{40} + d_{24})$ & $x^{12} + 2x^{10} - 2x^9 + 3x^8 - 2x^7 + 5x^6 - 2x^5 + 3x^4 - 2x^3 + 2x^2 + 1$\\ & $2x^{12} - 28x^{10} + 12x^9 + 6x^8 - 20x^7 + 74x^6 - 20x^5 + 6x^4 + 12x^3 - 28x^2 + 2$\\

$\frac{1}{9}(d_{43} + 15d_3)$ & $x^6 + x^3 + 1$ ; $x^6 - 4x^5 + 12x^3 - 4x+ 1$ \\

$\frac{1}{14}(d_{47} + 3 d_7)$ & $x^{10} - x^9 + 3x^8 - 3x^7 + 4x^6 - 4x^5 + 4x^4 - 3x^3 + 3x^2 - x + 1$ \\
                                & $2x^{10} - 2x^9 - 17x^8 + 3x^6 + 36x^5 + 3x^4 - 17x^2 - 2x + 2$ \\

$\frac{1}{6}(d_{51} + 2d_3)$ & $x^6 - 3x^5 + 6x^4 - 7x^3 + 6x^2 - 3x + 1$ ; $x^6 + 30x^5 + 12x^4 - 88x^3 + 12x^2$ \\&\qquad $+ 30x + 1$ \\

$\frac{1}{12}(d_{51} + 6d_{11})$ & $x^8 + 2x^7 + 2x^6 - x^4 + 2x^2 + 2x + 1 $\\& $  x^8 + 6x^7 + 15x^6 - 10x^5 - 42x^4 - 10x^3 + 15x^2 + 6x + 1$\\

$\frac{1}{9}(d_{51} + 2d_{19})$ & $x^8 + x^7 + x^6 + x^5 + x^4 + x^3 + x^2 + x + 1$ \\
                                & $2x^8 - 6x^7 - 33x^6 + 10x^5 + 72x^4 + 10x^3 - 33x^2 - 6x + 2$ \\

$\frac{1}{12}(d_{52} + 12d_4)$ & $x^6 + x^4 + x^2 + 1; x^6 + 6x^5 - x^4 - 20x^3 - x^2 + 6x + 1$\\

$\frac{1}{20}(d_{52} + d_{20})$ & $x^6 - x^5 + 2x^4 - 2x^3 + 2x^2 - x + 1; 2x^6 - 2x^5 + 2x^4 - 8x^3 + 2x^2 - 2x + 2$\\

$\frac{1}{21} (d_{55} + 10d_{7})$ & 
$x^{10} - 3x^9 + 6x^8 - 8x^7 + 9x^6 - 9x^5 + 9x^4 - 8x^3 + 6x^2 - 3x + 1 $\\& $ 2x^{10} - 6x^9 - 3x^8 + 16x^7 - 30x^6 + 44x^5 - 30x^4 + 16x^3 - 3x^2 - 6x + 2$\\

$\frac{2}{15} (d_{56}-32d_4)$  & $x^6 + 2x^5 + 3x^4 + 3x^3 + 3x^2 + 2x + 1; 2x^6 + 20x^5 + 54x^4 + 74x^3 + 54x^2 + 20x + 2$ \\

$ \frac{2}{21} (-d_{56} + 26d_{8} + 38d_{4} )$ & $x^{10} - x^7 - x^3 + 1 $\\& $  2x^{10} + 16x^9 + 48x^8 + 98x^7 + 144x^6 + 168x^5 + 144x^4 + 98x^3 + 48x^2 + 16x + 2$ \\

$\frac{1}{9}(d_{59} + 7d_3)$ & $x^{10} + 2x^9 + 3x^8 + x^7 - x^6 - 3x^5 - x^4 + x^3 + 3x^2 + 2x + 1$ \\
                             & $2x^{10} + 4x^9 - 30x^8 - 58x^7 + 31x^6 + 120x^5 + 31x^4 - 58x^3 - 30x^2 + 4x + 2$ \\

$\frac{1}{15}(d_{67} + 9 d_3)$ & $x^{10} - 2x^9 + 2x^8 - 2x^6 + 3x^5 - 2x^4 + 2x^2 - 2x + 1$ \\
                               & $2x^{10} - 4x^9 - 4x^8 + x^6 + 12x^5 + x^4 - 4x^2 - 4x + 2$ \\
                               
$\frac{1}{16}(d_{68} + 8d_4)$ & $x^6 + x^4 + x^2 + 1$ ; $2x^6 - 16x^5 - 2x^4 + 40x^3 - 2x^2 - 16x + 2$ \\

$\frac{1}{18}(d_{68} - 4d_{19} + 24d_3)$ & $x^6 - x^3 + 1; 2x^6 - 4x^5 - 3x^4 + 12x^3 - 3x^2 - 4x + 2$\\

$\frac{1}{54}(d_{84} + 120 d_{4})$ & $x^8 + x^7 + x^6 - x^5 - x^4 - x^3 + x^2 + x + 1 $\\& $ 2x^8 + 2x^7 + 2x^6 - 2x^5 - 14x^4 - 2x^3 + 2x^2 + 2x + 2$\\

$\frac{2}{21}(d_{84} - 18d_{7})$ & $x^6 - x^5 + x^4 - x^3 + x^2 - x + 1;2x^6 + 26x^5 + 58x^4 + 82x^3 + 58x^2 + 26x + 2$\\

$\frac{1}{36}(d_{84} + 6 d_{20})$ & $x^8 - 2x^7 + 3x^6 - 2x^5 + 2x^4 - 2x^3 + 3x^2 - 2x + 1 $\\& $ 2x^8 - 4x^7 + 4x^6 + 12x^5 - 32x^4 + 12x^3 + 4x^2 - 4x + 2$\\

$\frac{1}{45}(d_{87} + 4d_{15})$ & $x^{10} - x^9 + x^8 - 2x^7 + 2x^6 - x^5 + 2x^4 - 2x^3 + x^2 - x + 1$ \\
                                   & $2x^{10} - 2x^9 - 9x^8 + 20x^5 - 9x^2 - 2x + 2$ \\
                                   
$\frac{1}{15} (2 d_{88} - d_{68} + 16 d_{4})$ & $x^{16} - 4x^{15} + 9x^{14} - 12x^{13} + 10x^{12} - 4x^{11} + x^8 - 4x^5 + 10x^4 - 12x^3 + 9x^2 - 4x + 1$\\ & $2x^{16} - 8x^{15} - 18x^{14} + 56x^{13} + 52x^{12} - 104x^{11} - 128x^{10} + 16x^9$ \\ & $\qquad + \; 266x^8 + 16x^7 - 128x^6 - 104x^5 + 52x^4 + 56x^3 - 18x^2 - 8x + 2$ \\

$\frac{1}{36}(d_{107} + 21 d_3)$ & $x^{10} - x^9 + x^8 - x^6 + x^5 - x^4 + x^2-x+1$ \\ 
                                  & $2x^{10} - 2x^9 - 9x^8 - 4x^7 + 6x^6 + 16x^5 + 6x^4 - 4x^3 - 9x^2 - 2x + 2$ \\

$\frac{1}{90}(d_{111} + 6d_{15})$  & $x^6 + x^4 - x^3 + x^2 + 1; 2x^6 - x^4 - 8x^3 - x^2 + 2$ \\

$\frac{1}{15}(d_{120} - 84 d_{4})$  & $x^6 - x^5 - x + 1; 2x^6 + 34x^5 + 96x^4 + 136x^3 + 96x^2 + 34x + 2$\\

$\frac{1}{54}(d_{120} + 48d_{8})$  & $x^8 + x^7 + x^6 - x^5 - x^4 - x^3 + x^2 + x + 1 $ \\& $  x^8 - 2x^7 - 26x^6 + 2x^5 + 56x^4 + 2x^3 - 26x^2 - 2x + 1$\\

$\frac{1}{45}(d_{123} + 208d_3)$ &  
$x^{14} - 2x^{13} + 3x^{12} - 5x^{11} + 7x^{10} - 7x^9 + 8x^8 - 9x^7 + 8x^6 - 7x^5 + 7x^4 - 5x^3 + 3x^2 - 2x + 1$ \\
& $2x^{14} - 4x^{13} + 6x^{12} - 10x^{11} - 10x^{10} + 22x^9 - 35x^8 + 60x^7 - 35x^6 + 22x^5 - 10x^4$\\& $\qquad - 10x^3 + 6x^2 - 4x + 2$ \\

$\frac{1}{48}(d_{132} + 40 d_4)$  & $x^6 + 2x^5 + 4x^4 + 4x^3 + 4x^2 + 2x + 1; 2x^6 + 4x^5 - 10x^4 - 28x^3 - 10x^2 + 4x + 2$\\

$\frac{1}{60}(d_{136} + 2d_{56})$  & 
$x^{12} + x^{11} + 2x^{10} + x^9 - 2x^7 - 2x^6 - 2x^5 + x^3 + 2x^2 + x + 1$ \\
& $2x^{12} + 2x^{11} - 19x^{10} - 18x^9 - 4x^8 + 16x^7 + 50x^6 + 16x^5 - 4x^4 - 18x^3 - 19x^2 + 2x + 2$ \\

$\frac{1}{80}(d_{164} + 4 d_{20})$ & 
$x^{10} - x^9 + 2x^8 - 2x^7 + 3x^6 - 2x^5 + 3x^4 - 2x^3 + 2x^2 - x + 1$ \\
& $2x^{10} - 2x^9 - 15x^8 - 2x^7 + 5x^6 + 32x^5 + 5x^4 - 2x^3 - 15x^2 - 2x + 2$ \\

$\frac{1}{90}(d_{183} + 24 d_{7})$  & $x^{12} - 2x^{11} + 2x^{10} - x^8 + x^6 - x^4 + 2x^2 - 2x + 1$ \\ & $2x^{12} - 4x^{11} - 12x^{10} + 44x^9 - 33x^8 - 44x^7 + 96x^6 - 44x^5 - 33x^4 + 44x^3 - 12x^2 - 4x + 2$\\

$\frac{1}{90}(d_{183}+6d_{15})$  & 
$x^6 + x^4 + x^3 + x^2 + 1$; $x^6 + 12x^5 - 2x^4 - 32x^3 - 2x^2 + 12x + 1$ \\

$\frac{1}{90}(d_{195} + 20d_{15} + 104d_{3})$  & $x^{14} + 2x^{13} + 3x^{12} + 3x^{11} + 3x^{10} + x^9 - x^7 + x^5 + 3x^4 + 3x^3 + 3x^2 + 2x + 1$ \\& $2x^{14} + 4x^{13} - 6x^{12} - 42x^{11} - 69x^{10} - 22x^9 + 84x^8 + 148x^7 + 84x^6 - 22x^5 - 69x^4$\\ &$ \qquad - 42x^3 - 6x^2 + 4x + 2$\\

$\frac{1}{72}( d_{219}+36d_3)$  & 
$x^6 - x^5 + x^3 - x + 1$; $x^6 + 14x^5 - 32x^3 + 14x + 1$ \\

$\frac{1}{72}(d_{219}+12d_{11})$  & 
$x^{12} - 2x^{11} + 3x^{10} - 2x^9 + 2x^7 - 3x^6 + 2x^5 - 2x^3 + 3x^2 - 2x + 1$ \\ & $
2x^{12} - 4x^{11} - 14x^{10} + 48x^9 - 35x^8 - 48x^7 + 104x^6 - 48x^5 - 35x^4 + 48x^3 - 14x^2 - 4x + 2$\\

$\frac{1}{144}(d_{228} + 120d_{4})$ & $x^6 + 2x^4 + 2x^2 + 1; 2x^6 - 2x^4 - 12x^3 - 2x^2 + 2$\\

$\frac{1}{120}(d_{260}+112d_4)$  & 
$x^{12} + 2x^{11} + 2x^{10} + 2x^9 + 3x^8 + 2x^7 + x^6 + 2x^5 + 3x^4 + 2x^3 + 2x^2 + 2x + 1$ \\
& $2x^{12} + 4x^{11} - 16x^{10} - 60x^9 - 46x^8 + 68x^7 + 146x^6 + 68x^5 - 46x^4 - 60x^3 - 16x^2 + 4x + 2$ \\

$\frac{1}{108}(d_{264} + 6d_{24})$  & 
$x^{12} - x^{11} + 3x^{10} - 3x^9 + 4x^8 - 4x^7 + 4x^6 - 4x^5 + 4x^4 - 3x^3 + 3x^2 - x + 1$ \\
& $2x^{12} - 2x^{11} - 21x^{10} + 54x^9 - 28x^8 - 68x^7 + 134x^6 - 68x^5 - 28x^4 + 54x^3 - 21x^2 - 2x + 2$ \\

$\frac{1}{108}(d_{291}+36d_3)$  & 
$x^{10} + x^8 + x^7 + x^6 + x^5 + x^4 + x^3 + x^2 + 1$ \\
& $2x^{10} - 22x^8 - 22x^7 + 23x^6 + 56x^5 + 23x^4 - 22x^3 - 22x^2 + 2$ \\

$\frac{1}{165}(d_{303} + 66 d_{7})$  & 
$x^{12} - 3x^{11} + 5x^{10} - 4x^9 + 5x^7 - 7x^6 + 5x^5 - 4x^3 + 5x^2 - 3x + 1$ \\
& $2x^{12} - 6x^{11} - 13x^{10} + 66x^9 - 63x^8 - 64x^7 + 158x^6 - 64x^5 - 63x^4 + 66x^3 - 13x^2 - 6x + 2$ \\

$\frac{1}{210}(d_{555} + 112 d_{3})$  & 
$x^{10} + x^9 + 2x^8 + 3x^7 + 4x^6 + 3x^5 + 4x^4 + 3x^3 + 2x^2 + x + 1$ \\
& $2x^{10} + 2x^9 - 35x^8 - 54x^7 + 44x^6 + 132x^5 + 44x^4 - 54x^3 - 35x^2 + 2x + 2$ \\

$\frac{1}{432}(d_{804} + 216 d_4)$  & 
$x^{10} + 2x^9 + 4x^8 + 4x^7 + 5x^6 + 4x^5 + 5x^4 + 4x^3 + 4x^2 + 2x + 1$ \\
& $2x^{10} + 4x^9 - 44x^8 - 88x^7 + 58x^6 + 208x^5 + 58x^4 - 88x^3 - 44x^2 + 4x + 2$ \\
\bottomrule
\caption{Overview of permissible polynomials found in this work, other than those provided in Table~\ref{tab:new_strong}. The polynomial for $\frac{1}{72}(d_{219}+36d_3)$ already appeared in \cite[Example~8]{BRV2}, but the other examples seem to be new. The second column gives two polynomials $a(x)$ and $b(x)$ such that the polynomial $P(x,y) = a(x)(y^2+1)+b(x)y$ is permissible. The first column then gives the corresponding numerically inferred Mahler measure $\m(P)$.}
\label{tab:T1}
\end{longtable}
}
\clearpage

\begin{table}[ht!] 
\centering
\renewcommand{\arraystretch}{1.2}
\begin{tabular}{lcccccc} 
\toprule
$\ell$ & $2$ & $3$ & $4$ & $5$ & $6$ & $7$\\
\midrule
$\#$ $\{s=1\}$ & $10$ & $10$ & $20$ & $3$ & $12$ & $1$ \\
$\#$ $\{s>1\}$ & $3$ & $46$ & $28$ & $28$ & $62$ & $9$ \\
$B$ & $100$ & $100$ & $100$ & $30$ & $30$ & $10$ \\
\bottomrule
\end{tabular}
\caption{Output of our algorithm.}
\label{tab:output}
\end{table}

\bibliographystyle{amsplain}

\begin{thebibliography}{10}

\bibitem{BM}
M.-J.~Bertin and M.~Mehrabdollahei, \emph{An exact family of bivariate polynomials and {V}ariants of {C}hinburg's {C}onjectures}, arXiv:2407.20634, 22 pp., 2025. To appear in Annales mathématiques du Québec.

\bibitem{Boyd}
D.~W.~Boyd, \emph{Kronecker's theorem and {L}ehmer's problem for polynomials in several variables}, J. Number Theory \textbf{13} (1981), no.~1, 116--121.

\bibitem{Boyd-L}
D.~W.~Boyd, \emph{{Mahler's} measure and special values of {$L$}-functions}, Experiment. Math. \textbf{7} (1998), no.~1, 37--82.

\bibitem{BRV1}
D.~W.~Boyd and F.~Rodriguez-Villegas, \emph{Mahler's measure and the dilogarithm. {I}}, Canad. J. Math. \textbf{54} (2002), no.~3, 468--492.

\bibitem{BRV2}
D.~W.~Boyd and F.~Rodriguez-Villegas, \emph{Mahler's {M}easure and the {D}ilogarithm (II)}, arXiv:math/0308041, 37 pp., with an appendix by N.~M.~Dunfield, 2005.

\bibitem{BZ}
F.~Brunault and W.~Zudilin, \emph{Many variations of {M}ahler measures---a lasting symphony}, Australian Mathematical Society Lecture Series \textbf{28}, Cambridge University Press, Cambridge, 2020.

\bibitem{Chinburg}
T.~Chinburg, \emph{Mahler measures and derivatives of {$L$}-functions at non-positive integers}, unpublished manuscript, 9 pp., 1984.

\bibitem{Deninger}
C.~Deninger, \emph{Deligne periods of mixed motives, $K$-theory and the entropy of certain $\mathbf{Z}^n$-actions}, J. Amer. Math. Soc. \textbf{10} (1997), no.~2, 259--281.

\bibitem{GM}
A.~Guilloux and J.~March\'{e}, \emph{Volume function and {M}ahler measure of exact polynomials}, Compos. Math. \textbf{157} (2021), no.~4, 809--834.

\bibitem{LQ}
H.~Liu and H.~Qin, \emph{Mahler measure of families of polynomials defining genus 2 and 3 curves}, Exp. Math. \textbf{32} (2023), no.~2, 321--336.

\bibitem{LQ-code}
H.~Liu, and H.~Qin, \emph{Code for Mahler measure of polynomials defining genus 2 and 3 curves}, \url{https://github.com/liuhangsnnu/mahler-measure-of-genus-2-and-3-curves/}, accessed March 19th, 2026.

\bibitem{Ma}
V.~Maillot,
\emph{G\'eom\'etrie d'Arakelov des vari\'et\'es toriques et fibr\'es en droites int\'egrables},
M\'em. Soc. Math. Fr. (N.S.) \textbf{80} (2000).

\bibitem{Ray}
G.~A.~Ray, \emph{Relations between {M}ahler's measure and values of {$L$}-series}, Canad. J. Math. \textbf{39} (1987), no.~3, 694--732.


\bibitem{Smyth}
C.~J.~Smyth, \emph{A {K}ronecker-type theorem for complex polynomials in several variables}, Canad. Math. Bull. \textbf{24} (1981), no.~4, 447--452.

\bibitem{PARI}
The PARI~Group (Univ. Bordeaux), \emph{{PARI/GP version \texttt{2.17.2}}}, 2025. Available at \url{http://pari.math.u-bordeaux.fr/}.

\bibitem{Vandervelde}
S.~Vandervelde, \emph{The {M}ahler measure of parametrizable polynomials}, J. Number Theory \textbf{128} (2008), no.~8, 2231--2250.

\bibitem{Zudilin}
W.~Zudilin, \emph{Apéry limits and {M}ahler measures}, arXiv:2109.12972, 6 pp., 2021.

\end{thebibliography}

\end{document}